\documentclass[12pt]{amsart}
\usepackage{amssymb}
\usepackage{amsmath, amscd}
\usepackage{amsthm}
\usepackage{xcolor}

\newtheorem{theorem}{Theorem}[section]

\newtheorem{proposition}[theorem]{Proposition}
\newtheorem{lemma}[theorem]{Lemma}

\newtheorem{corollary}[theorem]{Corollary}
\theoremstyle{definition}

\newtheorem{remark}[theorem]{Remark}

\begin{document}

\author[J. Cui]{Jian Cui}
\address{School of Mathematics and Statistics, Anhui Normal University \\ Wuhu, Anhui 241002, China}
\email{cui368@ahnu.edu.cn}

\author[P. Danchev]{Peter Danchev}
\address{Institute of Mathematics and Informatics, Bulgarian Academy of Sciences \\ "Acad. G. Bonchev" str., bl. 8, 1113 Sofia, Bulgaria}
\email{danchev@math.bas.bg; pvdanchev@yahoo.com} \vspace{0.5cm}
\author[D. Jin]{Danya Jin}
\address{School of Mathematics and Statistics, Anhui Normal University \\ Wuhu, Anhui 241002, China}
\email{jindanya@126.com} \vspace{0.5cm} {\small \bf Jian Cui ~,~
Peter Danchev ~and~ Danya-Jin\\}

\title[Rings whose Nil-Clean and Clean Elements ...] {Rings whose Nil-Clean and Clean Elements are \\
Uniquely Nil-Clean}

\maketitle

\begin{abstract} We consider and study those rings in which each nil-clean or clean
element is uniquely nil-clean. We establish that, for abelian rings,
these rings have a satisfactory description and even it is shown
that the classes of abelian rings and the rings in which nil-clean
elements are uniquely nil-clean do coincide. Moreover, we prove that
the rings in which clean elements are uniquely nil-clean coincide
with the subclass of abelian rings consisting of only unipotent
units and, in particular, that in the semipotent case we have a
complete characterization only in terms of the former ring and its
divisions. Likewise, some extension properties and group rings for
such kinds of rings are also considered.
\end{abstract}

\bigskip

\section{\bf Introduction and Conventions}

All rings into consideration in the current work are associative and contain an identity element. Let $R$ be a ring. The symbols $U(R)$, $J(R)$, ${\rm nil}(R)$ and ${\rm idem}(R)$ stand for the group of all units, the Jacobson radical, the set of all nilpotents of $R$ and the set of all idempotents of $R$, respectively. For an arbitrary $n\in \mathbb{N}$, the ring consisting of all $n\times n$ matrices over $R$ and the ring consisting of all $n\times n$ upper triangular matrices over $R$ are, respectively, denoted by $\mathbb{M}_n(R)$ and $\mathbb{T}_n(R)$. Let $\mathbb{Z}_n$ be the ring of integers modulo $n$.

Mimicking \cite{Ni1}, an element $r$ of a ring $R$ is called {\it clean}, provided $r=u+e$, where $u\in U(R)$ and $e\in {\rm idem}(R)$. If, in addition, $ue=eu$, the element $r$ is said to be {\it strongly clean} (see \cite{Ni2}). Likewise, if in the record of $r$ there exist a unique unit $u$ and an idempotent $e$ with $r=u+e$, the element $r$ is then termed {\it uniquely clean} (see \cite{NZ} and also \cite{CNZ} in the case of group rings). Some other valuable sources in this topic are \cite{CWZ}, \cite{He} and \cite{KLNZ}, as well. Moreover, when each element of a ring $R$ is equipped with one of these properties, the ring $R$ also retains the same name. It was proved in \cite{NZ} that a ring $R$ is uniquely clean if, and only if, $R$ is abelian (i.e., all idempotents of $R$ are central), $R/J(R)$ is boolean and all idempotents of $R$ lift modulo $J(R)$.

On the other side, imitating these notions, in \cite{Di} were defined the kinds of {\it nil-clean}, {\it strongly nil-clean} and {\it uniquely nil-clean} elements of rings as follows: An element $r$ of a ring $R$ is called {\it nil-clean}, provided $r=q+e$, where $q\in {\rm nil}(R)$ and $e\in {\rm idem}(R)$. If, in addition, $qe=eq$, the element $r$ is said to be {\it strongly nil-clean}. Likewise, if in the record of $r$ there exist a unique nilpotent $q$ and an idempotent $e$ with $r=q+e$, the element $r$ is then termed {\it uniquely nil-clean}. It is worth mentioning that the uniqueness of the nil-clean record in the commutative case is also confirmed in \cite{DMc}. When each element of a ring $R$ is equipped with one of these properties, the ring $R$ also retains the same name. On this
vein, it is also worth noticing that arbitrary strongly nil-clean rings were independently characterized in \cite{DL} and \cite{KWZ} as those rings $R$ {\it for which $J(R)$ is nil and $R/J(R)$ is boolean}.

It is well known that, and also easy to check that, nil-clean rings are always clean, while the converse fails. However, for the element-wise case, the situation is totally different and unexpected. For a more detailed information in this way, we refer to \cite{WTZ}. Nevertheless, one can say that if $2\in {\rm nil}(R)$, then each nil-clean element is necessarily clean.

Incidentally, while studying rings with unipotent units (that are rings $R$ for which $U(R)=1+ {\rm nil}(R)$ and, for shortness, hereafter abbreviating them as {\it UU-rings}), it was shown in \cite[Theorem 3.7]{DL} that in UU-rings any element is (strongly) clean if, and only if, it is (strongly) nil-clean, respectively (see Corollary 3.11 in \cite{DL} too). On that vein, recently, in \cite{CZ}, it was provided a comprehensive exploration of those rings whose clean elements are uniquely clean, calling them {\it CUC-rings}. They are a proper generalization of uniquely clean rings. Specifically, the authors considered certain crucial properties of such rings including their natural extensions and group rings.

Our motivating tool to investigate such rings is that the uniqueness of nil-clean elements is definitely {\it not} so good studied as that of clean elements, bearing in mind the discussion presented above. Thus, our further plan is to examine those rings in which every nil-clean or clean element is uniquely nil-clean, calling them {\it NCUNC}-rings and {\it CUNC}-rings, respectively, and thus extending the class of uniquely nil-clean rings. Concretely, in that
aspect, our main establishments are structured thus: In Section 2, we prove the unexpected fact that the rings in which nil-clean elements possess an unique record are exactly the abelian rings (see Theorem~\ref{01}). In Section 3, we describe rings in which clean elements have an unique record as nil-clean elements and show that these are precisely the abelian rings whose units are unipotents (i.e., a sum of the identity and a nilpotent) (see
Theorem~\ref{3-01}). In the case of semipotent rings, we give an other valuable classification and also provide a close relation with the aforementioned CUC-rings (see Theorem~\ref{3-05} and Proposition~\ref{3-07}). In the final fourth section, we study group rings of CUNC-rings. For a ring $R$ and a group $G$, we write $RG$ for the group ring of $G$ over $R$. We establish some necessary and sufficient conditions by proving that $RG$ is a {\rm CUNC}-ring if,
and only if, $R$ is a {\rm CUNC}-ring and $G$ is a $2$-group, whenever $G$ is a locally finite group (see Theorem \ref{4-8}).

\medskip


\medskip

\section{\bf Nil-clean elements are uniquely nil-clean}

We begin here with the following useful technicality.

\begin{proposition} Let $R$ be a ring. Then the following are equivalent$:$
\par$(1)$ Every nilpotent of $R$ is uniquely nil-clean.
\par$(2)$ $({\rm nil}(R)+{\rm nil}(R))\cap {\rm idem}(R)=\{0\}$.
\end{proposition}

\begin{proof}
``(1) $\Rightarrow$ (2)". Let $e\in ({\rm nil}(R)+{\rm nil}(R))\cap
{\rm idem}(R)$. Then, $e=b+c$ for some $b,c\in {\rm nil}(R)$. So
$c=e+(-b)=0+c$ are two nil-clean decomposition of $c$. By
hypothesis, $e=0$, and hence $({\rm nil}(R)+{\rm nil}(R))\cap {\rm
idem}(R)=\{0\}$, as asserted.

``(2) $\Rightarrow$ (1)". For any $b\in {\rm nil}(R)$, it is clear
that $b=0+b$ is a nil-clean representation. Assume, for a moment,
that there exists another nil-clean decomposition, say $b=f+c$,
where $f\in {\rm idem}(R)$ and $c\in {\rm nil}(R)$. Then $f=b-c\in
({\rm nil}(R)+{\rm nil}(R))\cap {\rm idem}(R)=\{0\}$, and so $b$ is
uniquely nil-clean, as claimed.
\end{proof}

A ring $R$ is said to be a \emph{NCUNC-ring} if every nil-clean
element of $R$ is uniquely nil-clean. Recall once again that a ring
$R$ is \emph{abelian} if all idempotents of $R$ are central.
Besides, for shortness, we denote by $C_{k}^{n}$ the binomial
coefficient $\binom{n}{k}$, where $k\geq n\geq 1$ are integers.

\medskip

We now arrive at our first basic result, which is rather curious.

\begin{theorem}\label{01} Let $R$ be a ring. Then the following are equivalent$:$
\par$(1)$ Every nil-clean element of $R$ is uniquely nil-clean, that is, $R$ is a {\rm NCUNC}-ring.
\par$(2)$ Every strongly nil-clean element of $R$ is uniquely nil-clean.
\par$(3)$ Every idempotent of $R$ is uniquely nil-clean.
\par$(4)$ $R$ is an abelian ring.
\par$(5)$ Every idempotent of $R$ commutes with all nilpotents and $({\rm nil}(R)+{\rm nil}(R))\cap {\rm idem}(R)=\{0\}$.
\par$(6)$ $({\rm idem}(R)-{\rm idem}(R))\cap {\rm nil}(R)=\{0\}$.
\end{theorem}

\begin{proof} ``(1) $\Rightarrow$ (2)". This implication is obvious.

``(2) $\Rightarrow$ (3)". Clearly, each idempotent is strongly
nil-clean. Therefore, (2) follows by hypothesis.

``(3) $\Rightarrow$ (4)". Let $e^2=e\in R$. Thus, for any $r\in R$,
we obtain that $$e=(e-(1-e)re)+(1-e)re=(e-er(1-e))+er(1-e)$$ are two
nil-clean decompositions of $e$. So, we have $er(1-e)=(1-e)re$,
which implies that $er=ere=re$. Thus, $R$ is abelian.

``(4) $\Rightarrow$ (1)". Suppose that $e+b=f+c$ for some $e,f\in
{\rm idem}(R)$ and $b,c\in {\rm nil}(R)$. Then, $e-f=c-b$. As $R$ is
abelian, we know that $e,f,b$ and $c$ commute one another. So,
$$(e-f)^2=(c-b)(e-f)\in {\rm idem}(R)\cap {\rm nil}(R)=\{0\}.$$ It now follows
that $e=(2e-1)f$, whence $$e=(2e-1)e=(2e-1)^2f=f.$$ This ensures
that $R$ is a NCUNC-ring, as stated.

``(4) $\Rightarrow$ (5)". It is enough to show that $({\rm
nil}(R)+{\rm nil}(R))\cap {\rm idem}(R)=\{0\}$. To that goal, for
any $e\in ({\rm nil}(R)+{\rm nil}(R))\cap {\rm idem}(R)$, one has
that $e=b-c$ for some $b,c\in {\rm nil}(R)$. We may let $b^k=c^k=0$
with $k\geq 1$. Then, we deduce that
\begin{align*}
(e+c)^k&=e+C_k^1ec+\cdots +C_{k}^{k-1}ec^{k-1}+C_k^kc^k\\
&=e(1+C_k^1c+\cdots +C_{k}^{k-1}c^{k-1})\\
&=b^k=0.
\end{align*}
Note that $C_k^1c+\cdots +C_{k}^{k-1}c^{k-1}\in {\rm nil}(R).$ So,
we can easily verify that
$$1+C_k^1c+\cdots +C_{k}^{k-1}c^{k-1}\in U(R),$$ which assures $e=0$,
as required.

``(5) $\Rightarrow$ (6)". Take $$x\in ({\rm idem}(R)-{\rm
idem}(R))\cap {\rm nil}(R).$$ Then, there exist $e,f\in {\rm idem}(R)$
satisfying $x=e-f$. By hypothesis, we have $$e(1-f)=x-xf\in ({\rm
nil}(R)+{\rm nil}(R))\cap {\rm idem}(R)=\{0\}.$$ It now follows that
$$-x=-xf=(f-e)f=(1-e)f$$ must be zero, as needed.

``(6) $\Rightarrow$ (3)". Given $e^2=e\in R$. Clearly, $e$ is
nil-clean. Write $e=f+b$, where $f\in {\rm idem}(R)$ and $b\in {\rm
nil}(R)$. So, one has $$e-f=b\in ({\rm idem}(R)-{\rm idem}(R))\cap
{\rm nil}(R)=\{0\}.$$ Thus, $e=f$, which proves that $e$ is uniquely
nil-clean, as formulated.
\end{proof}

Now, according to Theorem \ref{01}, one sees that commutative rings, reduced rings and rings with
trivial idempotents are all NCUNC-rings.

\begin{remark} We claim that the condition ``$({\rm nil}(R)+{\rm nil}(R))\cap {\rm
idem}(R)=\{0\}$" in Theorem \ref{01}(5) does not imply that $R$ is
abelian. For example, set $R=\mathbb{T}_2(\mathbb{Z}_2)$. Then, one
inspects by simple matrix calculations that $$({\rm nil}(R)+{\rm
nil}(R))\cap {\rm idem}(R)=\{0\},$$ but $R$ is manifestly not
abelian.
\end{remark}

In view of ``(3) $\Leftrightarrow$ (4)" in Theorem \ref{01}, we
derive the following result immediately.

\begin{corollary}\label{02}
An idempotent $e$ of a ring $R$ is uniquely nil-clean if, and only if,
$e$ is central.
\end{corollary}

Next, owing to Theorem \ref{01}, we can derive a similar result to that from \cite{CZ}.

\begin{corollary}\label{03}
Let $I$ be a nil-ideal of a ring $R$. Then the following are
equivalent$:$
\par$(1)$ $R$ is an abelian ring.
\par$(2)$ $R/I$ is an abelian ring and idempotents lift uniquely modulo $I$.
\end{corollary}

\begin{proof} The proof surely can be extracted from \cite[Lemma 2.8(1)]{CZ}, but
we give an independent proof from the view point of the nil-cleanness and uniquely
nil-cleanness.

``(1) $\Rightarrow$ (2)". Let $\overline{e}^2=\overline{e}\in R/I$.
As $I$ is nil, idempotents lift modulo $I$ in the traditional
manner. So, we may let $e\in {\rm idem}(R)$. Notice that $e$ is
central, and so $\overline{e}$ is central too. Now, to prove that
idempotent lift uniquely modulo $I$, we assume that there exists
$f\in {\rm idem}(R)$ such that $e-f\in I$. Then, $e=f+b$ for some
$b\in I\subseteq {\rm nil}(R)$. However, in view of Corollary
\ref{02}, we obtain that $e=f$, as desired.

``(2) $\Rightarrow$ (1)". Given $e\in {\rm idem}(R)$. Let $e=f+b$ be
a nil-clean decomposition with $e\in {\rm idem}(R)$ and $b\in {\rm
nil}(R)$. Then, $\overline{e}=\overline{f}+\overline{b}$ is a
nil-clean decomposition of $\overline{e}\in R/I$. Since the quotient-ring $R/I$ is
abelian, it follows that $\overline{e}=\overline{f}$ by applying
Theorem \ref{01}. Hence, $e=f$, and thus $e$ is uniquely nil-clean in
$R$. Therefore, $R$ is abelian, as wanted.
\end{proof}

\medskip


\medskip

\section{\bf Clean elements are uniquely nil-clean}

It is well known that nil-clean rings are clean (see also
\cite{Di}), but there exist nil-clean elements which are not clean
(\cite{WTZ}). In this section, we study rings whose clean elements
are uniquely nil-clean. We say that a ring $R$ is a \emph{CUNC-ring}
if every clean element of $R$ is uniquely nil-clean. Recall that a
ring $R$ is known to be a \emph{UU-ring} (see, e.g., \cite{DL}) if $U(R)=1+{\rm nil}(R)$.

\medskip

We are now able to prove the following major assertion.

\begin{theorem}\label{3-01}
Let $R$ be a ring. Then the following are equivalent$:$
\par$(1)$ $R$ is a {\rm CUNC}-ring.
\par$(2)$ $R$ is an abelian {\rm UU}-ring.
\end{theorem}

\begin{proof} ``(1) $\Rightarrow$ (2)". Notice the well-known fact that all idempotents are clean. By
hypothesis, all idempotents are uniquely nil-clean. Furthermore, in
view of Theorem \ref{01}, we conclude that $R$ is abelian. Now,
given $u\in U(R)$, we elementarily see that $u$ is nil-clean since
it is clean. Thus, we write that $u=e+b$, where $e\in {\rm idem}(R)$
and $b\in {\rm nil}(R)$. From $R$ is abelian, one obtains that
$e=u-b\in U(R)$, and hence $e=1$. So, $R$ is $UU$, as asserted.

``(2) $\Rightarrow$ (1)". Suppose that $a\in R$ is clean. Write
$a=(1-e)+u$ with $e\in {\rm idem}(R)$ and $u\in U(R)$. As $R$ is a
UU-ring, it must be that $u=1+b$ for some $b\in {\rm nil}(R)$ and
$2\in {\rm nil}(R)$. It now follows that
$$a=(1-e)+u=e+(1-2e+u)=e+(2-2e+b),$$ where $n_1:=2-2e+b$ is
nilpotent since $R$ is abelian. This proves that $a=e+n_1$ is
nil-clean. If, however, there is another nil-clean decomposition
$a=f+n_2$, where $f\in {\rm idem}(R)$ and $n_1\in {\rm nil}(R)$,
then we get that $e-f=n_2-n_1$. By application of \cite[Corollary
2.4]{Ster}, ${\rm nil}(R)$ is a subring of $R$ as $R$ is a
$UU$-ring. So, we may write that $$e-f=(e-f)^3=(n_2-n_1)^3\in {\rm
nil}(R),$$ and then it follows at once that $e=f$. Therefore, $a$ is
uniquely nil-clean, as claimed.
\end{proof}

The next statement is a direct consequence of a simple combination of Theorems~\ref{01} and \ref{3-01}.

\begin{corollary}\label{3-06}
Let $R$ be a ring. Then the following are equivalent$:$
\par$(1)$ $R$ is a {\rm CUNC}-ring.
\par$(2)$ $R$ is abelian and every clean element is nil-clean.
\par$(3)$ $R$ is a {\rm NCUNC}-ring and every unit is nil-clean.
\end{corollary}

We now concentrate on the following reduction assertion.

\begin{corollary}\label{3-04}
Let $R$ be a ring. Then the following are equivalent$:$
\par$(1)$ $R$ is a {\rm CUNC}-ring.
\par$(2)$ $J(R)$ is nil, $R/J(R)$ is a {\rm CUNC}-ring and idempotents lift
uniquely modulo $J(R)$.
\end{corollary}

\begin{proof}
``(1) $\Rightarrow$ (2)". Utilizing Theorem \ref{3-01}, $R$ is an
abelian UU-ring. So, $R/J(R)$ is  UU, and $J(R)$ is nil by
consulting with \cite[Theorem 2.4(2)]{DL}. Furthermore, employing
Corollary \ref{03}, it follows that idempotents lift uniquely modulo
$J(R)$ and $R/J(R)$ is abelian. Now, applying Theorem \ref{3-01},
 $R/J(R)$ is a CUNC-ring.

``(2) $\Rightarrow$ (1)". Since $R/J(R)$ is a CUNC-ring, $R/J(R)$ is
an abelian UU-ring. By hypothesis and with Corollary \ref{03} at
hand, $R$ has to be abelian. Now, as $R/J(R)$ is a UU-ring and
$J(R)$ is nil, \cite[Theorem 2.4(2)]{DL} applies to conclude that
$R$ is a UU-ring, as expected.
\end{proof}

It is worthwhile to note the following: one readily verifies that
the ring $\mathbb{T}_2(\mathbb{Z}_2)$ is a UU-ring, but not a
CUNC-ring as it is not abelian and even more, for any ring $R$ and
any integer $n\geq 2$, both rings $\mathbb{M}_n(R)$ and
$\mathbb{T}_n(R)$ are {\it not} CUNC.

\begin{corollary}\label{3-08}
If $R$ is a {\rm CUNC}-ring, then the corner $eRe$ is a {\rm CUNC}-ring for any
$e^2=e\in R$.
\end{corollary}

\begin{proof}
It is clear that the subring of an abelian ring is still abelian. By assumption, $eRe$ is abelian. In view of Corollary \ref{3-06}, it suffices to show that every clean element of $eRe$ is also nil-clean. To that goal, assume
$a\in eRe$ is clean. Then, there exist $f^2=f\in eRe$ and $w\in U(eRe)$ such that $a=f+w$. So, one has  $$a+(1-e)=f+(w+(1-e))$$ is a clean decomposition in $R$. It thus follows that $a+(1-e)$ is nil-clean
in $R$. Hence, we write that $$a+(1-e)=g+b$$ with $g^2=g\in R$ and $b\in {\rm nil}(R)$. It now follows that $a=ege+ebe$ is a nil-clean element in the corner ring $eRe$, because $e$ is central, as required.
\end{proof}

The next statement is helpful and shows the inheritance by subrings and direct products of the property being a CUNC-ring. Note that our subrings need {\it not} be unital (herein, a subring of a given ring is {\it unital} whenever it contains the same identity element).

\begin{proposition}\label{3-03}
\par$(1)$ A subring $S$ of a {\rm CUNC}-ring $R$ is again a {\rm CUNC}-ring.
\par$(2)$ A direct product $\prod R_i$ is a {\rm CUNC}-ring if, and
only if, so is $R_i$ for each $i$.
\end{proposition}

\begin{proof} (1) One easily checks that any subring of an abelian ring is abelian as well. However, in
view of Theorem \ref{3-01}, it suffices to show only that $S$ is a
UU-ring. For convenience, we use $1_R$ and $1_S$ to denote the
identity of $R$ and $S$, respectively. To that goal, let $v\in
U(S)$. Then, $v+(1_R-1_S)\in U(R)$ with its inverse
$v^{-1}+(1_R-1_S)$. By assumption, we have $v+(1_R-1_S)=1_R+b$ for
some $b\in {\rm nil}(R)$, which in turn implies that $v=1_S+b$ and
$$b=v-1_S\in S\cap {\rm nil}(R)={\rm nil}(S).$$ Thus, $S$ is UU, as pursued.

(2) For the sufficiency, it is clear that $\prod R_i$ is abelian as
all $R_i$ are abelian. Let $(a_i)\in \prod R_i$ be a clean element.
Then, by plain element-wise arguments, it is evident that all
vector's components $a_i$ are clean. Referring to Corollary
\ref{3-06}, one gets that $a_i$ is nil-clean for each $i$, which in
turn insures that $(a_i)$ is nil-clean in $\prod R_i$. Hence, $\prod
R_i$ a CUNC-ring, as formulated. The necessity follows from (1).
\end{proof}

The next result contrasts to \cite[Corollary 2.13(1)]{CZ} and somewhat refines \cite[Corollary 2.14(1)]{CZ}.

\begin{proposition}
\par$(1)$ For any ring $R$, the power series ring $R[[x]]$ is not {\rm CUNC}.
\par$(2)$ A polynomial ring $R[x]$ over a commutative ring
$R$ is {\rm CUNC} if, and only if, so is $R$.
\end{proposition}

\begin{proof} (1) Note the principal fact that the Jacobson radical of $R[[x]]$ is not nil. Thus, in view
of Corollary \ref{3-04}, $R[[x]]$ is really {\it not} a {\rm CUNC}-ring.

(2) It is a well-known fact that ${\rm nil}(R[x])={\rm nil}(R)[x]$
(see, for instance, \cite{L}) whenever $R$ is commutative. First,
assume that $R$ is a {\rm CUNC}-ring. Let $f(x)=\sum_{i=0}^na_ix^i
\in U(R[x])$. Then, as it is known in the existing literature (see,
e.g., \cite{L}), it follows that $a_0\in U(R)$ and $a_i\in {\rm
nil}(R)$ for all $1\leq i \leq n$. By hypothesis, we can write
$a_0=1+b_0$ for some $b_0\in {\rm nil}(R)$. It thus follows that
$$f(x)=1+b_0+\sum_{i=1}^na_ix^i\in 1+{\rm nil}(R[x]).$$ Finally, Theorem \ref{3-01} applies to get that $R[x]$ is, indeed, a CUNC-ring, as claimed.

The converse claim follows with the aid of Proposition \ref{3-03}(1).
\end{proof}

A ring $R$ is known to be \emph{semipotent} if every one-sided ideal {\it not}
contained in $J(R)$ contains a non-zero idempotent. It is a
well-known fact that clean rings are semipotent (see, e.g.,
\cite{Ni1}).

\medskip

We now manage to prove the following chief result.

\begin{theorem}\label{3-05}
Let $R$ be a ring. Then the following are equivalent$:$
\par$(1)$ $R$ is a uniquely nil-clean ring.
\par$(2)$ $R$ is a semipotent {\rm CUNC}-ring.
\par$(3)$ $R$ is abelian, $J(R)$ is nil and $R/J(R)$ is boolean.
\end{theorem}

 \begin{proof} ``(1) $\Rightarrow$ (2)". It is clear that $R$ is a CUNC-ring. Since
uniquely nil-clean rings are nil-clean and so clean, we derive that $R$ is semipotent.

``(2) $\Rightarrow$ (3)". By virtue of Theorem \ref{3-01} and
Corollary \ref{3-04}, it suffices to show that $R/J(R)$ is boolean.
To this purpose, assume on the contrary. Then, there exists $a\in R$ such that
$a-a^2\not\in J(R)$. Since $R$ is semipotent, there is a non-zero
$e\in {\rm idem}(R)$ satisfying $e\in (a-a^2)R$. Therefore, $e=(a-a^2)t$
for some $t\in R$. As $e$ is central, we may write $$e=(ea)(e(1-a))(et),$$ which
yields that both $ea$ and $e(1-a)$ are units of the corner ring $eRe$. Writing now
$v=ea+(1-e)$, we then obtain $v\in U(R)$. But since $R$ is a UU-ring, it must be that
$$v-1=-e(1-a)\in {\rm nil}(R),$$ which is the wanted contradiction. Hence, $R/J(R)$ is
Boolean.

``(3) $\Rightarrow$ (1)". It is clear that Boolean rings are always
uniquely nil-clean. By assumption, $R/J(R)$ is a CUNC-ring. In
accordance with Corollary \ref{3-04}, we have $R$ is a CUNC-ring.
Now, to prove that $R$ is uniquely nil-clean, we only need to show
that each element of $R$ is clean. Indeed, the quotient $R/J(R)$ is a clean ring,
and idempotents lift modulo $J(R)$. Consequently, \cite[Proposition
6]{HN} allows us to deduce that $R$ is clean, as required.
\end{proof}

Combing Theorem \ref{3-05} with the corresponding results from \cite{Ni1}, we obtain the following result immediately.

\begin{corollary}
If $R$ is a {\rm CUNC}-ring, then the following are equivalent$:$
\par$(1)$ $R$ is a uniquely nil-clean ring.
\par$(2)$ $R$ is a nil-clean ring.
\par$(3)$ $R$ is a clean ring.
\par$(4)$ $R$ is an exchange ring.
\par$(5)$ $R$ is a semipotent ring.
\end{corollary}

We proceed by proving the next claim, which gives a relationship between CUC-rings from \cite{CZ} and our CUNC-rings in the case of semi-potentness.

\begin{proposition}\label{3-07} Suppose that $R$ is a semipotent ring. Then $R$ is a {\rm CUNC}-ring if, and only if, $R$ is a {\rm CUC}-ring and $J(R)$ is nil.
\end{proposition}

\begin{proof} Combining \cite[Theorem 3.2]{CZ} with Theorems~\ref{3-01} and \ref{3-05}, we are done.
\end{proof}

\medskip


\medskip

\section{\bf CUNC group rings}

Let $R$ be a ring, $G$ a group, and, as usual, we denote by $RG$ the group ring of $G$ over $R$.

In this section, we intend to establish a criterion for a group ring to be CUNC under some minimal restrictions on the former group and ring objects. Specifically, we are able to achieve this, provided the group is locally finite. Recall that a group $G$ is a \emph{$2$-group} if every its element has order that is a power of $2$. Also, we recollect that a group is \emph{locally finite} if each finitely generated subgroup is finite.

\medskip

The next two technicalities are helpful to us for proving the main result stated below.

\begin{lemma}\label{4-01} Suppose that $R$ is a {\rm CUNC}-ring. Then the following
hold$:$
\par$(1)$ $2\in {\rm nil}(R)$.
\par$(2)$ If $I$ is a nil-ideal of $R$, then $R/I$ is a {\rm CUNC}-ring.
\end{lemma}

\begin{proof}
$(1)$ follows by Theorem \ref{3-01} and \cite[Theorem 2.6(1)]{DL}.

$(2)$ Notice that, consulting with Theorem~\ref{3-01}, $R$ is both
abelian and UU. As $I$ is nil, by a similar argument to that in the implication ``(1)
$\Rightarrow$ (2)" from Corollary \ref{03}, the factor-ring $R/I$ is abelian. Further, in view
of \cite[Theorem 2.4(1)]{DL}, we deduce that the quotient $R/I$ a UU-ring. Thus,
$R/I$ is a {\rm CUNC}-ring by the usage of Theorem \ref{3-01}, as stated.
\end{proof}

\begin{lemma}\label{4-05}
Let $R$ be a ring and let $G$ be a group. If $RG$ is a {\rm CUNC}-ring, then $R$ is a {\rm CUNC}-ring and $G$ is a $2$-group.
\end{lemma}

\begin{proof}
As $R$ can be viewed as a subring of $RG$, we refer to Proposition
\ref{3-03}(1) to get that $R$ is a {\rm CUNC}-ring. Now, Lemma
\ref{4-01} applies to infer that $2$ is nilpotent in both $RG$ and
$R$. which implies that $2RG$ is a nil-ideal of $RG$, and so it
follows that $(R/2R)G\cong RG/2RG$ is a {\rm CUNC}-ring.
Furthermore, we may assume $2=0\in R$. Given any $g\in G$, we know
that $g\in U(RG)$. By assumption, $1-g$ is nilpotent in $RG$, say
$(1-g)^{2^m}=0$ for some $m\geq 1$. This guarantees that
$$g^{2^m}-1=(g-1)^{2^m}=0,$$ and, therefore, $g^{2^m}=1$. Thus, $G$
is a $2$-group, as required.
\end{proof}

We recall now that the map $f:RG\rightarrow R$, defined by
$f(\sum_{g\in G} a_g g)=\sum_{g\in G}a_g$, is a surjective ring
homomorphism. The kernel of $f$ is called the \emph{augmentation
ideal} of $RG$ and, standardly, is denoted by $\omega (RG)$. It is
well known that $\omega (RG)$ is an ideal of $RG$ generated by the
set $\{1-g:g\in G\}$.

\medskip

Our next technical claim asserts the following.

\begin{lemma}\label{4-06}
If $R$ is a {\rm CUNC}-ring and $G$ is a locally finite $2$-group, then $RG$ is a {\rm CUNC}-ring.
\end{lemma}

\begin{proof}
According to Lemma \ref{4-01}, we have $2\in {\rm nil}(R)$, whence
$2\in J(R)$. Hence, $R$ is abelian with $2\in J(R)$. Now,
\cite[Lemma 11]{CNZ} enables us that every idempotent of $RG$ lies
in $R$ since $G$ is a locally finite $2$-group, which means that
$RG$ is abelian. To prove now that $RG$ is a {\rm CUNC}-ring, as
pursued, we only need to show that $RG$ is a UU-ring. Indeed, as $G$
is a locally finite $2$-group and $2\in {\rm nil}(R)$, the ideal
$\omega (RG)$ is nil owing to \cite[Proposition 16(ii)]{Co}.
However, we observe that $R\cong RG/\omega (RG)$ is a {\rm
CUNC}-ring. Finally, \cite[Theorem 2.4(1)]{DL} allows us to conclude
that $RG$ is a UU-ring, as desired.
\end{proof}

Now, combining Lemma \ref{4-05} with Lemma \ref{4-06}, we can extract the following chief result of this section.

\begin{theorem}\label{4-8}
Let $R$ be a ring and let $G$ be a locally finite group. Then $RG$ is a {\rm CUNC}-ring if, and only if, $R$ is a {\rm CUNC}-ring and $G$ is a $2$-group.
\end{theorem}

It is worthwhile noticing that a valuable non-trivial necessary and sufficient condition for a group ring to be UU was found in \cite{DA}. Notice also that Theorem~\ref{3-01} is a guarantor that CUNC-rings are always abelian UU-rings, and vice versa.

\medskip

We close our investigation with the following two intriguing and non-trivial questions, which are closely related to our presentation given above.

\medskip

\noindent{\bf Problem 1.} Examine those rings whose nil-clean elements are either uniquely clean or uniquely strongly clean.

\medskip

\noindent{\bf Problem 2.} Examine those rings whose strongly clean elements are either uniquely clean or uniquely strongly clean.

\medskip
\medskip
\medskip

\centerline{\bf FUNDING}

The first-named author of this research paper was supported by Anhui Provincial Natural Science Foundation (No. 2008085MA06), Key Laboratory of Financial Mathematics of Fujian Province University (Putian University) (No. JR202203) and the Key project of Anhui Education Committee (No. gxyqZD2019009).

The second-named author of this research paper was partially supported by the Bulgarian National Science Fund under Grant KP-06 No. 32/1 of December 07, 2019, as well as by the Junta de Andaluc\'ia, Grant FQM 264, and by the BIDEB 2221 of T\"UB\'ITAK.


\vskip3.0pc

\begin{thebibliography}{99}

\bibitem{CZ}
G. C$\check{a}$lug$\check{a}$reanu and Y. Zhou, {\it Rings whose clean elements are uniquely clean}, Mediterr. Math. J., 2023: {\bf 20}(1), paper 15.

\bibitem{CNZ}
J. Chen, W.K. Nicholson and Y. Zhou, {\it Group rings in which every element is uniquely the sum of a unit and an
idempotent}, J. Algebra, 2006: {\bf 306}, 453--460.

\bibitem{CWZ}
J. Chen, Z. Wang and Y. Zhou, {\it Rings in which elements are uniquely the sum of an idempotent and a unit that commute}, J. Pure Appl. Algebra, 2009: {\bf 213}(2), 215--223.

\bibitem{Co}
I.G. Connell, {\it On the group ring}, Canad. J. Math., 1963: {\bf 15}, 650--685.

\bibitem{DA}
P.V. Danchev and O. Al-Mallah, {\it UU group rings}, Eurasian Bull. Math., 2018: {\bf 1}(3), 94--97.

\bibitem{DL}
P.V. Danchev and T.-Y. Lam, {\it Rings with unipotent units}, Publ. Math. (Debrecen), 2016: {\bf 88}(3-4), 449--466.

\bibitem{DMc}
P.V. Danchev and W.Wm. McGovern, {\it Commutative weakly nil clean unital rings}, J. Algebra, 2015: {\bf 425}, 410--422.

\bibitem{Di}
A.J. Diesl, {\it Nil clean rings}, J. Algebra, 2013: {\bf 383}, 197--211.

\bibitem{HN}
J. Han and W.K. Nicholson, {\it Extensions of clean rings}, Commun. Algebra, 2001: {\bf 29}(6), 2589--2595.

\bibitem{He}
S. Hegde, {\it On uniquely clean elements}, Commun. Algebra, 2023: {\bf 51}(5), 1835--1839.

\bibitem{KLNZ}
D. Khurana, T.-Y. Lam, P.P. Nielsen and Y. Zhou, {\it Uniquely clean elements in rings}, Commun. Algebra, 2015: {\bf 43}(5), 1742--1751.

\bibitem{KWZ}
M.T. Ko\c{s}an, Z. Wang and Y. Zhou, {\it Nil-clean and strongly nil-clean rings}, J. Pure  Appl. Algebra, 2016: {\bf 220}(2), 633--646.

\bibitem{L}
T.-Y. Lam, A First Course in Noncommutative Rings, Second Edition, Graduate Texts in Math., Vol. {\bf 131}, Springer-Verlag, Berlin-Heidelberg-New York, 2001.

\bibitem{Ni1}
W.K. Nicholson, {\it Lifting idempotents and exchange rings}, Trans. Amer. Math. Soc., 1977: {\bf 229}, 269--278.

\bibitem{Ni2}
W.K. Nicholson, {\it Strongly clean rings and Fitting's lemma}, Commun. Algebra, 1999: {\bf 27}(8), 3583--3592.

\bibitem{NZ}
W.K. Nicholson and Y. Zhou, {\it Rings in which elements are uniquely the sum of an idempotent and a unit}, Glasg. Math. J., 2004: {\bf 46}, 227--236.

\bibitem{Ster}
J. $\check{S}$ter, {\it Rings in which nilpotents form a subring}, Carpathian J. Math., 2016: {\bf 32}(2), 251--258.

\bibitem{WTZ}
Y. Wu, G. Tang, G. Deng and Y. Zhou, {\it Nil-clean and unit-regular elements in certain subrings of
$\mathbb{M}_2(\mathbb{Z})$}, Czechoslovak Math. J., 2019: {\bf 69(144)}(1), 197--205.

\end{thebibliography}
\end{document}